\documentclass[10pt]{amsart}

\usepackage{float}
\usepackage{amscd,amsmath,amssymb,amsfonts,amsthm, ascmac, bm}
\usepackage{enumerate, comment, mathtools} 

\usepackage{caption}
\usepackage{subcaption}
\usepackage{hyperref}
\usepackage{relsize}

\usepackage{booktabs}
\usepackage{a4wide}

\usepackage[mathscr]{euscript}

\usepackage{ragged2e}
\usepackage{ytableau,varwidth}

\usepackage{tikz-cd}
\usepackage{tikz}
\usetikzlibrary{snakes, 
	3d, matrix, decorations.pathreplacing,calc,decorations.pathmorphing, fit, patterns, trees, decorations.markings, arrows, positioning}
\usetikzlibrary{calc}

\usepackage{diagbox}

\numberwithin{equation}{section}
\allowdisplaybreaks[1]

\usepackage{colonequals} 

\allowdisplaybreaks

\DeclareMathOperator{\Cone}{Cone}

\DeclareMathOperator{\Gr}{Gr}

\newcommand{\C}{{\mathbb{C}}}
\newcommand{\R}{{\mathbb{R}}}

\newcommand{\Z}{{\mathbb{Z}}}

\renewcommand{\P}{\mathrm{P}}
\DeclareMathOperator{\flag}{Fl}

\newcommand{\hati}[1]{\textcolor{gray}{{\hat{s}_{{#1}}}}}

\theoremstyle{plain}
\newtheorem{theorem}{Theorem}[section]
\newtheorem{lemma}[theorem]{Lemma}
\newtheorem{proposition}[theorem]{Proposition}
\newtheorem{corollary}[theorem]{Corollary}

\theoremstyle{definition}
\newtheorem{example}[theorem]{Example}

\theoremstyle{remark}
\newtheorem{remark}[theorem]{Remark}

\newcommand{\vred}{\textcolor{red}{|}}

\hypersetup{colorlinks}
\hypersetup{
	unicode=false,          
	pdftoolbar=true,        
	pdfmenubar=true,        
	pdffitwindow=false,     
	pdfstartview={FitH},    
	pdftitle={My title},    
	pdfauthor={Author},     
	pdfsubject={Subject},   
	pdfcreator={Creator},   
	pdfproducer={Producer}, 
	pdfkeywords={keyword1} {key2} {key3}, 
	pdfnewwindow=true,      
	colorlinks=true,       
	linkcolor=blue,          
	citecolor=red,        
	filecolor=violet,      
	urlcolor=violet       
}

\begin{document}

\author{Shin-young Kim}
\address[Shin-young Kim]{Basic Science Research Institute, Ewha Womans University, Seoul 03760, Republic of Korea}
\email{shinyoung.kim@ewha.ac.kr}

\author{Eunjeong Lee}
\address[Eunjeong Lee]{Department of Mathematics,
	Chungbuk National University,
	Cheongju 28644, Republic of Korea}
\email{eunjeong.lee@chungbuk.ac.kr}

\thanks{Lee is the corresponding author.\\
Lee was supported by the research grant of the Chungbuk National University in 2022. The first author Kim was supported by Basic Science Research Program through the National Research Foundation of Korea(NRF) funded by the Ministry of Education(2021R1A6A1A10039823).
}

\title{Gorenstein toric Schubert varieties in Grassmannians}

\date{\today}



\begin{abstract} 
A partial flag variety is a smooth projective homogeneous variety admitting an action of a maximal torus $T$.
Schubert varieties are $T$-invariant subvarieties of the partial flag varieties.
We study toric Schubert varieties in Grassmannian varieties with respect to the action of the torus $T$. Indeed, we present an explicit description of the fan of a Gorenstein toric Schubert variety in a Grassmannian, and we prove that any Gorenstein toric Schubert variety in a Grassmannian variety is Fano.
\end{abstract}

\maketitle

\setcounter{tocdepth}{1} 
\tableofcontents

\section{Introduction} 
Let $G$ be a semisimple Lie group over $\C$, $T$ a maximal torus of $G$, and $B$ a Borel subgroup of $G$ containing $T$. Let $P$ a parabolic subgroup which is a closed subgroup containing the Borel subgroup $B$. The homogeneous space $G/P$ is a smooth projective algebraic variety, called a \emph{\textup{(}partial\textup{)} flag variety} or a rational homgeneous manifold. The natural $G$-action on a partial flag variety $G/P$ inherits an action of the maximal torus $T$ which is defined by the left multiplication of $T$. There are finitely many $T$-fixed points which are indexed by the set~$W^P = W/W_P$ of minimal length left coset representatives, where $W$ is the Weyl group of $G$ and $W_P$ is the subgroup of $W$ determined by $P$. Schubert varieties are subvarieties of $G/P$ indexed by $W^P$. Indeed, for $w \in W^P$, the \emph{Schubert variety} $X_w$ is defined by the $B$-orbit closure at the $T$-fixed point $wP$,
$$X_w = \overline{B wP/P} \subset G/P.$$
The geometry of flag varieties and Schubert varieties provides interesting avenues connecting the geometry, combinatorics, and representation theory as exhibited in Schubert calculus (see, for example, \cite{Brion_Lecture, Fulton97Young}).

A Schubert variety $X_w$ inherits the $T$-action from that on $G/P$. We say that a Schubert variety~$X_w$ is \emph{toric} if it is a toric variety with respect to the $T$-action. 
When we choose a parabolic subgroup as a Borel subgroup $B$, that is when we consider $G/B$, there are some known results on toric Schubert varieties in $G/B$. For example, it is known from~\cite{Karu13Schubert} that a Schubert variety $X_w$ in $G/B$ is toric if and only if $w$ consists of distinct simple reflections. The description of the fan of a toric Schubert variety in $G/B$ is given in~\cite{Fan98} (also, see~\cite{GK94Bott}). 

In this paper, we study toric Schubert varieties in a Grassmannian $\Gr(d,n)$ with $G=\rm{SL}_n (\mathbb C)$. 
Recently, Hodges and Lakshmibai~\cite{HodgesLakshmibai22} provide a classification of spherical Schubert varieties in the Grassmannian $\Gr(d,n)$, and as a corollary, 
one can determine whether a Schubert variety $X_w$ in the Grassmannian $\Gr(d,n)$ is toric in terms of~$w$. Given $w \in W^P$, associating a partition $\lambda_w \colonequals (\lambda_1,\dots,\lambda_d)$ where $\lambda_i = w(d-i+1)-(d-i+1)$, we reinterpret their expression in terms of \emph{partitions}. More precisely, the Schubert variety $X_w \subset \Gr(d,n)$ is toric with respect to the action of the maximal torus $T$ if and only if the corresponding partition $\lambda_w$ is a hook-shaped partition in the $d \times (n-d)$ rectangle (see Proposition~\ref{prop_classify_toric_Schubert}).

We notice that not every toric Schubert variety in a Grassmannian $\Gr(d,n)$ is smooth or Gorenstein. Considering the known results~\cite{BL20Singular, WooYong06Gorenstein} on the smoothness and the Gorenstein condition for Grassmannian Schubert varieties, we can decide when a given toric Schubert variety $X_w$ is smooth or Gorenstein in terms of the partition~$\lambda_w$ (see Lemmas~\ref{prop_toric_smoothness} and~\ref{prop_toric_Gorenstein}). 
As the main result, we prove that the anticanonical divisor of a Gorenstein toric Schubert variety in Grassmannian is ample. 
\begin{theorem}\label{thm_intro_Fano}
Any Gorenstein toric Schubert variety in a Grassmannian is Fano.
\end{theorem}

To obtain the above theorem, we use the projection $\pi \colon G/B \to G/P$ induced by the inclusion $B \subset P$. Moreover, using the description of the fan of the toric Schubert variety in $G/B$, we obtain the fan of a Gorenstein toric Schubert variety in a Grassmannian in Proposition~\ref{prop_fan_in_Gr}. We notice that any smooth toric Schubert variety in $\Gr(d,n)$ is a projective space while any toric Schubert variety in $G/B$ is smooth.

The organization of this paper is given as follows. 
In Section~\ref{section_preliminaries}, we recall background on flag varieties and its Schubert varieties, especially for a Grassmannian $\Gr(d,n)$. In Section~\ref{sec_toric_Schubert_in_Gr}, we provide a description of toric Schubert varieties in $\Gr(d,n)$ in terms of the partition~$\lambda_w$. Moreover, we recall the fan of toric Schubert varieties in $G/B$ and relate it with the fan of toric Schubert varieities in $G/P$. In Section~\ref{section_Gorenstein_toric_Schubert}, we study Fano conditions on Gorenstein toric Schubert varieties, and we prove Theorem~\ref{thm_intro_Fano}.

\section{Preliminaries: Schubert varieties}\label{section_preliminaries}
In this section, we recall the flag varieties and Schubert varieties from~\cite{Brion_Lecture, BL20Singular}. Also, we consider the geometry of Grassmannian Schubert varieties, such as smoothness, Gorenstein conditions, and isomorphism classifications by recalling from~\cite{BL20Singular}, \cite{WooYong06Gorenstein}, and \cite{TX23isomGrassmannianSchubert}, respectively. 

Let $G$ be a semisimple Lie group, $T$ a maximal torus, and $B$ a Borel subgroup. Let $P$ be a parabolic subgroup which is a closed subgroup of $G$ containing the Borel subgroup $B$. Let $\Phi$ be the root system of $(G,T)$, $\Phi^+$ the set of positive roots, and $\Delta$ the set of simple roots. Then there is a bijective correspondence between the set of parabolic subgroups of $G$ and the set of subsets of~$\Delta$. Indeed, for $J \subset \Delta$, we denote by $\Phi_J^+$ the subset of $\Phi$ consisting of positive linear combination of simple roots in $J$. Then there exist a Parabolic subgroup $P_J$ with the Levi-decomposition $P_J=L_JU_J$ such that $\alpha \in \Phi_J^+$ are the characters of $L_J$. Conversely, if a parabolic subgroup $P$ is given, we can choose the set of simple root of $G$ whose restriction to $P$ is a character of $P$. We note that the parabolic subgroup $P_{\emptyset}$ is the Borel subgroup $B$.

In this paper, we concentrate on $G/P$ in the case when $G = \textrm{SL}_{n}(\C)$, which is a simple Lie group of type $A_{n-1}$. For a parabolic subgroup $P$, the homogeneous space $G/P$ is a smooth projective algebraic variety, called a \emph{\textup{(}partial\textup{)} flag variety}. 
When  $P = B$, the flag variety can be identified with the set of \emph{full flags} in $\C^n$:
\[
G/B \cong \{ V_{\bullet} = (\{0\} \subsetneq V_1 \subsetneq V_2 \subsetneq \cdots \subsetneq V_{n-1} \subsetneq \C^n) \mid \dim_{\C} V_i = i \quad \text{ for all } i =1,\dots,n-1\}.
\]
Moreover, when $P = P_d$ is a maximal parabolic subgroup corresponds to the set~$\Delta \setminus \{\alpha_d\}$, then the flag variety can be identified with the set of vector subspaces of dimension $d$ in~$\C^n$, called the \emph{Grassmannian variety} $\Gr(d,n)$:
\[
G/P_d \cong \Gr(d,n) := \{V \subset \C^n \mid \dim_{\C}V = d\}.
\]

Let $G/P$ be a flag variety. The left multiplication of $T$ on $G$ induces an action of $T$ on $G/P$. 
There is a bijective correspondence between the set of $T$-fixed points in $G/P$ and $W/W_{P}$, where $W$ is the Weyl group of~$G$ and $W_{P} = \langle s_{\alpha} \mid \alpha \in J, P=P_J \rangle$ (see, for example,~\cite[Section~2.6]{BL20Singular} or \cite[Proposition~1.1.1]{Brion_Lecture}). 
We notice that the Weyl group $W$ is generated by the elements $s_i \colonequals s_{\alpha_i}$ associated to the simple roots $\alpha_i$. We call $s_i$ a \emph{simple reflection} or a \emph{simple transposition}. Each element $w \in W$ can be written as a product of generators $w = s_{i_1} s_{i_2} \cdots s_{i_m}$. 
The \emph{length}~$\ell(w)$ of $w$ is the minimal number $m$ among all such expressions for $w$, and the word $s_{i_1} \cdots s_{i_m}$ is called a \emph{reduced word} (or \emph{reduced decomposition}) for~$w$. 
Let $W^P:=W/W_{P}$ be the set of minimal length coset representatives. 
For each $w \in W^P$, we denote by $w P \in G/P$ the corresponding fixed point. 
We have the Bruhat decomposition of $G/P$: 
\[
G/P = \bigsqcup_{w \in W^P} BwP/P,
\]
where $BwP/P$ is an affine cell of dimension $\ell(w)$, which is called a \emph{Bruhat cell}.

For $w \in W^P$, the \emph{Schubert variety} $X_w$ is defined to be the $B$-orbit closure $\overline{BwP/P}$ at the $T$-fixed point $wP$.
Since $T \subset B$, the Schubert variety $X_w$ is $T$-stable with respect to the left multiplication. Moreover, there is a bijective correspondence between the set of $T$-fixed points in $X_w$ and the set $\{v \in W^P \mid v \leq w\}$ (cf.~\cite[Section~1.2]{Brion_Lecture}).  Here, we compare two minimal length representatives using the \emph{Burhat order}, that is, for $w$ with a reduced expression  $w = s_{i_1} \cdots s_{i_m}$, we have $u \leq w$ if and only if there exists a reduced expression 
\[
u = s_{i_{j_1}} \cdots s_{i_{j_k}} \quad 1 \leq j_1 < \cdots < j_k \leq m.
\]

When $P=P_d$, the set $W^P$ consists of the identity and permutations having one descent at $d$.
\[
W^P = \{w \in \mathfrak{S}_n \mid 
w(1)<w(2)<\dots<w(d), \quad 
w(d+1)<w(d+2)<\dots<w(n)\}.
\]
The permutations of a Grassmannian variety are called \emph{Grassmannian permutations}. 
Throughout the paper, to denote a permutation, we will use both a reduced expression and the one-line notation:
\[
w = w(1) w(2) \ \cdots \ w(n).
\] 
In the case when we would like to emphasis the place where the descent appears, we will put  $\vred$. 
\begin{example}\label{example_Gr24_permutations}
Suppose that $n = 4$ and $d = 2$. Then 
\[
\begin{split}
W^P &= \{ w\in \mathfrak{S}_4\mid w(1)<w(2),\quad w(3) <w(4)\} \\
&=\{1234, 13\vred24, 14\vred23, 23\vred14, 24\vred13, 34\vred12\}.
\end{split}
\]
Considering the lengths of permutations in $W^P$, we obtain 
\[
G/P_2 = \Gr(2,4) = \C^0 \sqcup \C^1 \sqcup \C^2 \sqcup \C^2 \sqcup \C^3 \sqcup \C^4.
\]
\end{example}

To describe the smoothness condition on the Schubert varieties in Grassmannian $\Gr(d,n)$, we recall the Young diagram representation from~\cite[\S 9.3]{BL20Singular}. Given $w \in W^P$, associate a partition $\lambda_w \colonequals (\lambda_1,\dots,\lambda_d)$, where $\lambda_i = w(d-i+1) - (d-i+1)$. 
We say that a partition $\lambda$ consists of $r$ rectangles $p_1 \times q_1,\dots,p_r \times q_r$ for $p_i,q_i \geq 1$ if 
\[
\lambda = (p_1^{q_1},\dots,p_r^{q_r}) = (\underbrace{p_1,\dots,p_1}_{q_1 \text{ times}},\dots,\underbrace{p_r,\dots,p_r}_{q_r \text{ times}}). 
\]
We notice that $v \leq w$ in Bruhat order if and only if $\lambda_v \leq \lambda_w$. Here, we use the lexicographic order to compare partitions, that is, for partitions $\lambda =  (\lambda_1,\dots,\lambda_d)$ and $\mu = (\mu_1,\dots,\mu_d)$, we define~$\lambda \leq \mu$ if $\lambda_i \leq \mu_i$ for all $i=1,\dots,d$. 

Moreover, Gorenstein Schubert varieties also can be distinguished considering the corresponding Young diagrams as in~\cite{WooYong06Gorenstein}. 
We notice that for a Grassmannain permutation $w$ for $\Gr(d,n)$, the corresponding partition $\lambda$ sitting inside a $d \times (n-d)$ rectangle. Considering the lower border of partition as a \emph{lattice path} from the lower left-hand corner to the upper right-hand corner of the rectangle $d \times (n-d)$, a \emph{corner} is a lattice point on this path with lattice points of path both directly above and directly to the left of it. For instance, for $w = 2457136$ in $\Gr(4,7)$, we have $\lambda_w = (3,2,2,1)$. In the following diagram, we depict the lattice path in red and corners in blue.
\begin{center}
\begin{tikzpicture}[inner sep=0in,outer sep=0in]
	\node (n) {
		\begin{varwidth}{5cm}{
				\begin{ytableau}
					*(gray!30)~ & *(gray!30)~&*(gray!30)~ ~ \\
					*(gray!30)~ & *(gray!30)~& ~ \\
					*(gray!30)~ & *(gray!30)~& ~ \\
					*(gray!30)~ & ~& ~ \\
	\end{ytableau}}\end{varwidth}};
	\draw[thick,red] (n.south west)--
	([xshift=0.545cm]n.south west) 	node[draw=none, shape = circle,fill = blue, scale= 3] {} --
	([xshift=0.545cm, yshift=0.545cm]n.south west)
	--
	([xshift=1.09 cm, yshift = 0.545cm]n.south west) 	node[draw=none, shape = circle,fill = blue, scale= 3] {}--
	([xshift=-0.545cm, yshift = -0.545cm ]n.north east)
	 --
	([yshift=-0.545cm]n.north east) 	node[draw=none, shape = circle,fill = blue, scale= 3] {}--
	(n.north east);
\end{tikzpicture}
\end{center}

\begin{proposition}[{see \cite[Corollary~9.3.3]{BL20Singular} for the smoothness; see \cite{WooYong06Gorenstein} for the Gorenstein condition; see~\cite{TX23isomGrassmannianSchubert} for the isomorphism classification}]\label{prop_smoothness}
	For a Grassmannian Schubert variety $X_w$, the following statements hold.
\begin{enumerate}
\item 	$X_w$ is smooth if and only if $\lambda_w$ is the empty partition or consists of one rectangle. 
\item $X_w$ is Gorenstein if and only if all of the corners of $\lambda_w$ lie on the same antidiagonal. 
\item Let $X_{v}$ be Schubert variety which might be in different Grasmannian. Then $X_w$ and $X_v$ are isomorphic as varieties if and only if $\lambda_w$ and $\lambda_v$ are identical up to a transposition.
\end{enumerate}
\end{proposition}
We demonstrate the above proposition in the following example for $\Gr(2,4)$. 
\begin{example}
	Suppose that $n = 4$ and $d = 2$. Following Example~\ref{example_Gr24_permutations}, we have six permutations in $W^P$. The corresponding partitions and the lattice paths are given in the following.
	\begin{center}
		\begin{tabular}{c|cccccc}
			\toprule 
			$w$ & $1234$ & $1324$ & $1423$ & $2314$ & $2413$ & $3412$ \\
			\midrule 
			$\lambda_w$ & $\emptyset$ & $(1)$ & $(2)$ & $(1^2)$ & $(2,1)$ & $(2^2)$ \\[1em]
			& \begin{tikzpicture}[inner sep=0in,outer sep=0in]
				\node (n) {
					\begin{varwidth}{5cm}{
							\begin{ytableau}
								~ & ~ \\
								~ & ~
				\end{ytableau}}\end{varwidth}};
				\draw[thick,red] (n.south west)--
				([yshift=1.09cm]n.south west) 
					 --
				([xshift= -1.09cm]n.north east) --
				(n.north east);
			\end{tikzpicture}
			
			& \begin{tikzpicture}[inner sep=0in,outer sep=0in]
				\node (n) {
					\begin{varwidth}{5cm}{
							\begin{ytableau}
								*(gray!30)~ & ~ \\
								~ & ~
				\end{ytableau}}\end{varwidth}};
				\draw[thick,red] (n.south west)--
				([yshift=0.545cm] n.south west)--
				([xshift=0.545cm, yshift=0.545cm] n.south west) node[draw=none, shape = circle,fill = blue, scale= 3] {}--
				([xshift=-0.545cm] n.north east)--
				(n.north east);
			\end{tikzpicture}
			& \begin{tikzpicture}[inner sep=0in,outer sep=0in]
				\node (n) {
					\begin{varwidth}{5cm}{
							\begin{ytableau}
								*(gray!30)~ & *(gray!30)~ \\
								~ & ~
				\end{ytableau}}\end{varwidth}};
				\draw[thick,red] (n.south west)--
				([yshift=0.545cm] n.south west) --
				([xshift=0.545cm, yshift=0.545cm] n.south west)--
				([yshift=-0.545cm] n.north east) node[draw=none, shape = circle,fill = blue, scale= 3] {}--
				(n.north east);
			\end{tikzpicture}
			& \begin{tikzpicture}[inner sep=0in,outer sep=0in]
				\node (n) {
					\begin{varwidth}{5cm}{
							\begin{ytableau}
								*(gray!30)~ & ~ \\
								*(gray!30)~ & ~
				\end{ytableau}}\end{varwidth}};
				\draw[thick,red] (n.south west)--
				([xshift=0.545cm] n.south west)  node[draw=none, shape = circle,fill = blue, scale= 3] {} --
				([xshift=0.545cm, yshift=0.545cm] n.south west)--
				([xshift=-0.545cm] n.north east)--
				(n.north east);
			\end{tikzpicture}
			& \begin{tikzpicture}[inner sep=0in,outer sep=0in]
				\node (n) {
					\begin{varwidth}{5cm}{
							\begin{ytableau}
								*(gray!30)~ & *(gray!30) ~ \\
								*(gray!30)~ & ~
				\end{ytableau}}\end{varwidth}};
				\draw[thick,red] (n.south west)--
				([xshift=0.545cm] n.south west)  node[draw=none, shape = circle,fill = blue, scale= 3] {} --
				([xshift=0.545cm, yshift=0.545cm] n.south west) --
				([yshift=-0.545cm] n.north east) node[draw=none, shape = circle,fill = blue, scale= 3] {} --
				(n.north east);
			\end{tikzpicture}
			& \begin{tikzpicture}[inner sep=0in,outer sep=0in]
				\node (n) {
					\begin{varwidth}{5cm}{
							\begin{ytableau}
								*(gray!30)~ & *(gray!30) ~ \\
								*(gray!30)~ & *(gray!30) ~
				\end{ytableau}}\end{varwidth}};
				\draw[thick,red] (n.south west)--
				([xshift=0.545cm] n.south west)--
				([xshift=1.09cm] n.south west) node[draw=none, shape = circle,fill = blue, scale= 3] {}--
				([yshift=-0.545cm] n.north east)  --
				(n.north east);
			\end{tikzpicture}	\\
		smoothness & Yes & Yes & Yes & Yes & No & Yes \\
		Gorenstein & Yes & Yes  & Yes  & Yes  & Yes  & Yes \\
			\bottomrule 
		\end{tabular}
	\end{center}
	Applying Proposition~\ref{prop_smoothness}, the Schubert variety $X_{2413}$ is singular and the other Schubert varieties are all smooth in $\Gr(2,4)$. 
	Moreover, all Scubert varieties are Gorenstein. Finally, two Schubert varieties $X_{1423}$ and $X_{2314}$ are isomorphic since the corresponding Young diagrams are identical up to a transposition.
\end{example}

\section{Description of toric Schubert varieties in Grassmannians}\label{sec_toric_Schubert_in_Gr}

In this section, we consider \emph{toric} Schubert varieties in a Grassmannian $G/P_d=\Gr(d,n)$. To be more precise, we consider characterization of the toric Schubert varieties by recalling~\cite{HodgesLakshmibai22} and describe the fan of the toric Schubert varieties in $G/P_d$ by considering the corresponding toric Schubert varieties in the full flag variety $G/B$.

Throughout this section, we denote by $W^P$ the set of Grassmannian permutations in $\mathfrak{S}_n$ having one descent at $d$, that is, $W^P$ parametrizes Bruhat cells in the Grassmannian $\Gr(d,n)$. 
We first consider the classification on toric Schubert varieties $X_w$ in $\Gr(d,n)$ in terms of the partitions $\lambda_w$ and one-line notations of $w$. To do so, we need to introduce hook-shaped partitions. A partition~$\lambda$ is called \emph{hook-shaped} if it is of the form $(x, 1^y)$ for some $x \geq 1$ and $y\geq 0$. Here, $1^0$ (that is, when $y=0$) by mean that $1$ appears $0$ times. The number $(x-1)$ is called the \emph{arm-length}, which is the number of boxes in the hook $(x, 1^y)$ to the right of the top left box. Moreover, the number $y$ is called the \emph{leg-length}, which is the number of boxes in the hook $(x,1^y)$ to the below of the top leftmost box.

\begin{proposition}\label{prop_classify_toric_Schubert}
	For $w \in W^P$, let $X_w$ be the Schubert variety in the Grassmannian $\Gr(d,n)$. Then, the following statements are equivalent.
	\begin{enumerate}
		\item The Schubert variety $X_w$ is a toric variety with respect to the action of the maximal torus~$T$. 
		\item Either $w = e$ or a reduced decomposition of $w$ is 
\begin{equation}\label{eq_w_reduced}
	w = s_{d+x-1} s_{d+x-2} \cdots s_{d+1} s_{d-y} s_{d-y+1} \cdots s_d
\end{equation}
for $1 \leq x \leq n-d$ and $0 \leq y \leq d-1$. 
Here, if $x = 1$, then we set $s_{d+x-1} s_{d+x-2} \cdots s_{d+1} = e$.	
		\item The corresponding partition $\lambda_w$ is the empty partition or a hook-shaped partition $(x, 1^y)$ in the $d \times (n-d)$ rectangle. 
	\end{enumerate}
\end{proposition}
To provide a proof of Proposition~\ref{prop_classify_toric_Schubert}, we recall from~\cite{HodgesLakshmibai22} the classification of toric Schubert varieties (with respect to the action of $T$) in $\Gr(d,n)$. 
\begin{lemma}[{see~\cite[Collorary~6.0.2]{HodgesLakshmibai22}}]\label{lemma_HL_toric}
	 The Grassmannian Schubert variety $X_w$ is toric if and only if $w$ is the identity or its one-line notation is one of the following forms:
	\begin{enumerate}
		\item $w = 1 \ \cdots \ p \ p+2 \ \cdots \ d \ f \ \vred \ p+1 \ d+1 \ \cdots \ f-1 \ f+1 \cdots n$ for some integers $0\leq p < d-1$ and $d < f \leq n$.
		\item $w = 1 \ \cdots \ d-1 \ f \ \vred \ d \ \cdots f-1 \ f+1 \ \cdots n$ for some integer $d < f \leq n$.
	\end{enumerate}
	Here, we put $\vred$ the place where a descent appears.
\end{lemma}
\begin{proof}[Proof of Proposition~\ref{prop_classify_toric_Schubert}]
We first consider $(1) \iff (2)$. 
By Lemma~\ref{lemma_HL_toric}, we already know one-line notations of permutations describing toric Schubert varieties. 
For $w$ in the first case in Lemma~\ref{lemma_HL_toric}, we obtain 	\[
\begin{split}
	w &= 1 \ \cdots \ p \ p+2 \ \cdots \ d \ f \ \vred \ p+1 \ d+1 \ \cdots \ f-1 \ f+1 \cdots n \\
	&= s_{f-1} s_{f-2} \cdots s_{d+1} s_{p+1} s_{p+2} \cdots s_d, 
\end{split}
\]
By setting $x = f-d$ and $y = d-p-1$, the numbers $x$ and $y$ vary $0 < x \leq n-d$ and $0 < y \leq d-1$, and moreover, we obtain the reduced decomposition of $w$ in~\eqref{eq_w_reduced}. 
For $w$ in the second case in Lemma~\ref{lemma_HL_toric}, we obtain 
\[
\begin{split}
	w &= 1 \ \cdots \ d-1 \ f \ \vred \ d \ \cdots f-1 \ f+1 \ \cdots n \\
	&= s_{f-1} s_{f-2} \cdots s_d,
\end{split}
\]
By setting $x = f-d$, we obtain the reduced decomposition of $w$ in~\eqref{eq_w_reduced} (with $y=0$). 
This proves $(1) \iff (2)$.

Now we claim $(1) \iff (3)$.
To prove this claim, we consider the one-line notations in Lemma~\ref{lemma_HL_toric} again. If $w = e$, then we obtain $\lambda_w = \emptyset$ and the converse is also true. So, let us assume that $w \neq e$.

If $\lambda_w=(x,1^y)$ is given for $1\leq x \leq n-d$ and $0 \leq y \leq d-1$, then we have $$w(d)=x+d.$$
In addition, if $y=0$ and set $x+d=f$, then 
\begin{eqnarray*}
     w&=&w(1)w(2)\cdots w(d) \vred w(d+1)\cdots w(n) \\
     &=&1 \ \cdots \  d-1 \  f \ \vred \ d \ \cdots \ f-1 \ f+1 \ \cdots \ n  
\end{eqnarray*}
for $d+1 \leq f \leq n$.
If $y \neq 0$, set $x+d=f$ and $p+2=d-y+1$, then $$w(d-a)=1+d-a$$ for $0 \leq a \leq y$, and
\begin{eqnarray*}
     w&=&w(1)w(2)\cdots w(d)\vred w(d+1)\cdots w(n) \\
     &=&1 \ \cdots p \ p+2 \cdots \  f \ \vred \ p+1 \ d+1 \cdots \ f-1 \ f+1 \ \cdots \ n  
\end{eqnarray*}
for $d+1 \leq f \leq n$ and $0 \leq p < d-1$. Hence, by Lemma~\ref{lemma_HL_toric}, $(3)$ implies $(1)$.

We claim that $(1)$ implies $(3)$. For $w$ in the first case in Lemma~\ref{lemma_HL_toric}, by setting $x = f-d$ and $y = d-p-1$,  we obtain 
\begin{eqnarray*}\label{eq_lambda_w_1}
\lambda_w &=& (f-d, d-(d-1),\dots,(p+2)-(p+1))  \\
          &=& (f-d,1^{d-p-1}) \\
          &=& (x, 1^y),
\end{eqnarray*}
where the numbers $x$ and $y$ vary $0 < x \leq n-d$ and $0 < y \leq d-1$.
Specifically, if $f = d+1$, that is $x=1$, then
	\begin{equation*}\label{eq_lambda_w_2}
	\lambda_w = (1^{1+y}) \quad \text{ for } 0 < y \leq d -1.
	\end{equation*}
Accordingly, we obtain \emph{any} nonempty column partition in the $d \times (n-d)$ rectangle, that is, any hook-shaped partition in the $d \times (n-d)$ rectangle whose arm-length is zero. 
Moreover, if $ f > d+1$ and if we set $x = f-d$ and $y = d-p-1$, then we have 
\begin{equation}\label{eq_lambda_w_3}
\lambda_w = (x, 1^{y}) \quad \text{ for } 1 < x\leq n-d,~~ 0 < y\leq d-1.
\end{equation}
Therefore, we obtain \emph{any} hook having at least one arm-length and at least one leg-length.

For $w$ in the second case in Lemma~\ref{lemma_HL_toric}, by setting  $f-d = x$, we obtain 
\begin{equation*}\label{eq_lambda_w_4}
\lambda_w = (x) \quad \text{ for } 0 <  x\leq n-d.
\end{equation*}
Therefore, we obtain \emph{any} nonempty row partition in the $d \times (n-d)$ rectangle, that is, any hook-shaped partition in the $d \times (n-d)$ rectangle whose leg-length is zero.
Hence, the result follows. 
\end{proof}
\begin{remark}
	It is known from~\cite{Karu13Schubert} that a Schubert variety $X_w$ in the full flag variety is toric if and only if a reduced decomposition of $w$ consists of distinct simple reflections. Since the projection map $G/B \to G/P$ induced by the inclusion $B \subset P$ is $T$-equivariant, a Schubert variety $X_w$ in $G/P$ is toric (with respect to $T$) if $w$ consists of distinct simple reflections. 
\end{remark}

To describe the fan of the toric Schubert variety $X_w$, we use the fibration $\pi \colon G/B \to G/P$ induced by the inclusion $B \subset P$. To avoid confusion, we denote by $X_w^B$ the Schubert variety in $G/B$ corresponding to $w$. 
One can see that a reduced word in each case consists of distinct simple reflections by Proposition~\ref{prop_classify_toric_Schubert}. Accordingly, the corresponding toric Schubert variety in the full flag variety is again toric by~\cite{Karu13Schubert}. 

Any toric Schubert variety in $G/B$ is smooth and its fan can be described in terms of the Cartan integers (cf.~\cite{Fan98} and \cite{Karu13Schubert} and see Proposition~\ref{prop_fan_of_toric_in_full} below). Here, the \emph{Cartan integers} $c_{i,j}$ are entries of the Cartan matrix, which are given by 
\[
c_{i,j} = \begin{cases}
	2 & \text{ if  } i = j, \\
	-1 & \text{ if }|i-j| = 1, \\ 
	0 & \text{ otherwise}
\end{cases}
\]
because we are considering the case when $G$ is of Lie type $A$. 
Considering the proof of Theorem~4.23 in~\cite{LMP_Handbook}, we obtain the following statement:
\begin{proposition}[{see \cite[\S3.7]{GK94Bott}, \cite[Theorem~4.23]{LMP_Handbook}}]\label{prop_fan_of_toric_in_full}
	Let $w = s_{i_1} \cdots s_{i_m}$ be a reduced decomposition of $w \in \mathfrak{S}_n$. Suppose that $i_1,\dots,i_m$ are distinct. Then the fan of the toric Schubert variety $X_{w}^B$ in $G/B$ is isomorphic to the fan in $\R^m$ such that primitive ray vectors are the $2m$ column vectors of the following matrix
\begin{equation}\label{equation_ray_vectors}
\left[
\begin{array}{cccc|cccc}
	1 & 0 & \cdots & 0 & -1 & 0 & \cdots & 0 \\
	0 & 1 &  \cdots & 0 & & -1 & \cdots & 0 \\
	\vdots & & \ddots & & & -c_{i_j,i_k} & \ddots & \\
	0 & 0 & \cdots & 1 & & & & -1
\end{array}
\right] =: [\mathbf{v}_1 \ \dots \ \mathbf{v}_m \ \mathbf{w}_1 \ \dots \ \mathbf{w}_m],
\end{equation}
where the $(j,k)$ entry for $m \geq j >k \geq 1$ in the right submatrix above is $-c_{i_j,i_k}$ and $c_{i,j}$ are Cartan integers. Here, we denote by $\mathbf{v}_i, \mathbf{w}_i$ the column vectors. For each $\mathscr{J} = \{1 \leq j_1 < \cdots < j_k \leq m\} \subseteq [m]$, we have a maximal cone $C_v^B$ corresponding to the fixed point $v \colonequals s_{i_{j_1}} \cdots s_{i_{j_k}}$ as follows:
\[
C_v^B = \Cone( \{ \mathbf{v}_{i} \mid i \notin \mathscr{J} \} \cup \{\mathbf{w}_i \mid i \in \mathscr{J} \}).
\]
Here, we set $[m] \colonequals \{1,\dots,m\}$ for $m > 0$.
\end{proposition}

\begin{example}\label{example_2314_in_full_flag}
	Let $w  = 2314 = s_1s_2$. 
	Following Proposition~\ref{prop_fan_of_toric_in_full}, the fan of the Schubert variety~$X_{w}^B$ in the full flag variety has four rays whose ray generators are given as the column vectors of the following matrix:
	\[
	\begin{bmatrix}
		1 & 0 & -1 & 0 \\
		0 & 1 & 1 & -1
	\end{bmatrix}
	=: [\mathbf{v}_1 \ \mathbf{v}_2 \ \mathbf{w}_1 \ \mathbf{w}_2].
	\]
	Here, we denote by $\mathbf{v}_i, \mathbf{w}_i$ the ray generators. 
	There are four maximal cones corresponding to fixed points $eB, s_1B, s_2B, s_1s_2B$. In the following table, we display fixed points, the corresponding subsets $\mathscr{J}$ and the maximal cones. Also, see Figure~\ref{fig_fan_2314}.
	\begin{center}
	\begin{tabular}{c|cccc}
		\toprule 
		$v$ & $e$ & $s_1$ & $s_2$ & $s_1s_2$ \\
		\midrule 
		$\mathscr{J}$ & $\emptyset$ & $\{1\}$ & $\{2\}$ & $\{1,2\}$ \\
		$C_v^B$ & $\Cone(\mathbf{v}_1,\mathbf{v}_2)$
		&  $\Cone(\mathbf{w}_1,\mathbf{v}_2)$ 
		&  $\Cone(\mathbf{v}_1,\mathbf{w}_2)$
		& $\Cone(\mathbf{w}_1,\mathbf{w}_2)$\\
		\bottomrule 
	\end{tabular}
	\end{center}
\end{example}
\begin{figure}
	\begin{subfigure}[t]{0.4\textwidth}
		\centering
	\begin{tikzpicture}
		\fill[yellow!20] (0,0)--(2,0)--(2,2)--(0,2)--cycle;
		\fill[orange!20] (0,0)--(0,2)--(-2,2)--cycle;
		\fill[blue!15] (0,0)--(-2,2)--(-2,-2)--(0,-2)--cycle;
		\fill[green!15] (0,0)--(0,-2)--(2,-2)--(2,0)--cycle;

		\draw[gray,->] (-2,0)--(2,0);
		\draw[gray, ->] (0,-2)--(0,2);
		\draw[very thick,->] (0,0)--(1,0) node[at end, above] {$\mathbf{v}_1$};
		\draw[very thick, ->] (0,0)--(0,1) node[at end, right] {$\mathbf{v}_2$};
		\draw[very thick, ->] (0,0)--(-1,1) node[at end, above] {$\mathbf{w}_1$};
		\draw[very thick, ->] (0,0)--(0,-1) node[at end, right] {$\mathbf{w}_2$};
		
		\node[blue] at (1.2,1.2) {$eB$};
		\node[blue] at (-0.6, 1.7) {$s_1B$};
		\node[blue] at (1.2, -1.2) {$s_2 B$};
		\node[blue] at (-1.2, -1.2) {$s_1s_2B$};
	\end{tikzpicture}
	\caption{The fan of $X_{2314,B}$}\label{fig_fan_2314}
\end{subfigure}
\begin{subfigure}[t]{0.45\textwidth}
			\centering
	\begin{tikzpicture}
		\fill[yellow!20] (0,0)--(2,0)--(2,2)--(-2,2)--cycle;
		\fill[blue!15] (0,0)--(-2,2)--(-2,-2)--(0,-2)--cycle;
		\fill[green!15] (0,0)--(0,-2)--(2,-2)--(2,0)--cycle;
		
		\draw[gray,->] (-2,0)--(2,0);
		\draw[gray, ->] (0,-2)--(0,2);
		\draw[very thick,->] (0,0)--(1,0) node[at end, above] {$\mathbf{v}_1$};
		\draw[very thick, ->] (0,0)--(-1,1) node[at end, above] {$\mathbf{w}_1$};
		\draw[very thick, ->] (0,0)--(0,-1) node[at end, right] {$\mathbf{w}_2$};
		
		\node[blue] at (1.2,1.2) {$eB$};
		\node[blue] at (1.2, -1.2) {$s_2 B$};
		\node[blue] at (-1.2, -1.2) {$s_1s_2B$};
	\end{tikzpicture}
	\caption{The fan of $X_{2314}$}\label{fig_fan_2314_Gr}
\end{subfigure}
	\caption{The fan of $X_{2314}$ in $G/B$ or $G/P$}
\end{figure}
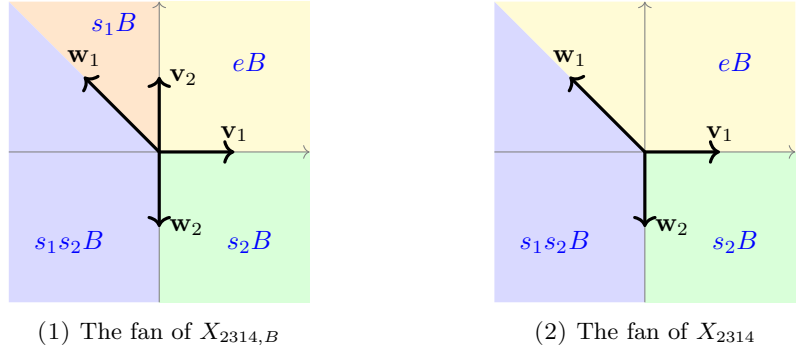

\begin{example}\label{example_2413_in_full_flag}
Let $w = 2413 = s_1s_3s_2$. 
Following Proposition~\ref{prop_fan_of_toric_in_full}, the fan of the Schubert variety~$X_{w}^B$ in the full flag variety has six rays whose ray generators are given as the column vectors of the following matrix:
\[
\begin{bmatrix}
	1 & 0 & 0 & -1 & 0 & 0 \\
	0 & 1 & 0 & 0 & -1 & 0 \\
	0 & 0 & 1 & 1 &  1 & -1 
\end{bmatrix} =: [\mathbf{v}_1 \ \mathbf{v}_2 \ \mathbf{v}_3 \ \mathbf{w}_1 \ \mathbf{w}_2 \ \mathbf{w}_3 ].
\]
Here, we denote by $\mathbf{v}_i, \mathbf{w}_i$ the column vectors. Then there are eight maximal cones as in Table~\ref{table_maximal_cones_2413}. 
\begin{table}
	\begin{tabular}{c|cccc}
		\toprule 
		$v$ & $e$ & $s_1$ & $s_3$ & $s_2$\\
		\midrule 
		$\mathscr{J}$ & $\emptyset$ & $\{1\}$ & $\{2\}$ & $\{3\}$\\
		$C_v^B$ & $\Cone(\mathbf{v}_1,\mathbf{v}_2,\mathbf{v}_3)$ & $\Cone(\mathbf{w}_1,\mathbf{v}_2,\mathbf{v}_3)$ &
		$\Cone(\mathbf{v}_1,\mathbf{w}_2,\mathbf{v}_3)$ & $\Cone(\mathbf{v}_1,\mathbf{v}_2,\mathbf{w}_3)$\\
		\hline 
		\hline 
		$v$ & $s_1s_3$ & $s_1s_2$ & $s_3s_2$ & $s_1s_3s_2$ \\
		\midrule 
		$\mathscr{J}$ & $\{1,2\}$ & $\{1,3\}$ & $\{2,3\}$ & $\{1,2,3\}$ \\
		$C_v^B$
		&	$\Cone(\mathbf{w}_1,\mathbf{w}_2,\mathbf{v}_3)$ 
		& 	$\Cone(\mathbf{w}_1,\mathbf{v}_2,\mathbf{w}_3)$
		& $\Cone(\mathbf{v}_1,\mathbf{w}_2,\mathbf{w}_3)$ 
		&	$\Cone(\mathbf{w}_1,\mathbf{w}_2,\mathbf{w}_3)$ 		\\
		\bottomrule 
\end{tabular}
\caption{Maximal cones in the fan of $X_{2413} \subset G/B$.}\label{table_maximal_cones_2413}
\end{table}
\end{example}

Since the fibration $\pi \colon G/B \to G/P$ is $T$-equivariant sending $wB \mapsto wP$ for any $T$-fixed point $wB \in G/B$, we obtain the following. 
\begin{proposition}\label{prop_fan_in_Gr}
	Let $w \in W^P$ defining a toric Schubert variety in $G/P$. For $v \in W^P$ such that $v \leq w$, the corresponding maximal cone $C_v$ in the fan of the toric Schubert variety $X_w$ in $G/P$ is given as follows:
	\[
	C_v = \bigcup_{u \in [v]} C_u^B,
	\]
	where $[v] := \{ u \in \mathfrak{S}_n \mid u \leq w,~~ u W_P = v W_P \}$. 
\end{proposition}

\begin{proof}
	Let $\Sigma_w^B$ be the fan of the toric Schubert variety $X_w^B$ in $G/B$ and let $\Sigma_w$ be that of $X_w$ in~$G/P$. Note that two Schubert varieties have the same dimension, say $k$. 
	Moreover, the restriction $\pi|_{X_w^B}$ of the projection $\pi \colon G/B \to G/P$ to $X_w^B$ defines a toric morphism $\phi \colonequals \pi|_{X_w^B} \colon X_w^B \to X_w$ since $\pi$ is a $T$-equivariant map sending $BwB/B$ to $BwP/P$, and we have $X_w^B = \bigcup_{v \leq w} BvB/B$ and $X_w = \bigcup_{v \leq w} BvP/P$.

	Since $\phi$ is a toric morphism, $\phi$ induces a $\Z$-linear map $\overline{\phi}$ that is compatible with the fans $\Sigma_w^B$ and $\Sigma_w$ by Theorem~3.3.4 in~\cite{CLS11Toric}. 
	Moreover, since the map $\phi$ is $T$-equivariant, the map $\overline{\phi} \colon \Z^k \to \Z^k$ is the identity. Accordingly, for each maximal cone $C_u^B \in \Sigma_w^B$, there exists a cone $C_v \in \Sigma_w$, corresponding to the fixed point $vP$, such that $\overline{\phi}_{\R}(C_u^B) = C_u^B \subseteq C_v$. Here, the cone $C_v$ containing $C_u^B$ can be specified by considering the corresponding torus orbits by~\cite[Lemma~3.3.21]{CLS11Toric}. By applying the Orbit-Cone correspondence of toric varieties (see~\cite[Theorem~3.2.6]{CLS11Toric}), it is enough to consider the image of the fixed points under the map $\phi$ to obtain $C_v$, that is, 
	\[
	C_u^B \subseteq C_v \iff \phi(uB) = vP.
	\]
	This implies 
	\[
	C_v = \bigcup_{\substack{u \leq w, \\ \phi(uB) = vP}} C_u^B. 
	\]
	Moreover, since $\phi(uB) = vP$ if and only if $u W_P = v W_P$, the result follows. 
\end{proof}

\begin{example}
	Consider $\Gr(2,4)$ and let $w = 2314$. Then there are three Grassmannian permutations $e, s_2, s_1s_2$ which are less than or equal to $w$. For each permutation, we have 
	\[
	[e] = \{e, s_1\}, \quad  
	[s_2] = \{s_2\}, \quad [s_1s_2] = \{s_1s_2\}.
	\]
	Accordingly, the fan of $X_{2314}$ is given as in Figure~\ref{fig_fan_2314_Gr} which defines the complex projective plane~$\C \P^2$. 
\end{example}

\section{Fano conditions on Gorenstein toric Schubert varieties}\label{section_Gorenstein_toric_Schubert}
In this section, we consider Gorenstein toric Schubert varieties in Grassmannian $\Gr(d,n)$ and prove Theorem~\ref{thm_intro_Fano} by providing descriptions of their fans explicitly. We denote by $W^P$ the set of Grassmannian permutations having one descent at $d$ as in the previous section. 

We first consider the smoothness and Gorenstein conditions on toric Schubert varieties. 
\begin{lemma}\label{prop_toric_smoothness}
	For a toric Schubert variety $X_w$, the following statements are  equivalent. 
	\begin{enumerate}
		\item $X_w$ is smooth. 
		\item A reduced decomposition of $w$ is the identity $e$ or it is given in the one of the following two forms. 
				\begin{enumerate}
			\item $w = s_{d+x-1} s_{d+x-2} \cdots s_d$ \quad for $1 < x \leq n-d$;  
			\item $w = s_{d-y} s_{d-y+1} \cdots s_d$ \quad for $0 \leq y \leq d-1$.
		\end{enumerate}
		\item The partition $\lambda_w$ is a column partition or a row partition. 		
	\end{enumerate}
\end{lemma}
\begin{proof}
	Suppose that $X_w$ is a toric Schubert variety. 
	By Proposition~\ref{prop_classify_toric_Schubert}, the partition $\lambda_w$ is a hook-shaped partition of the form $(x, 1^y)$ for $1 \leq x \leq n-d$ and $0\leq y \leq d-1$. By Proposition~\ref{prop_smoothness}(1), we obtain that $X_w$ is smooth if and only if $\lambda_w = (x)$ or $(1^y)$ which proves the equivalence $(1) \iff (3)$. Moreover, again by Proposition~\ref{prop_classify_toric_Schubert}, we obtain $(2) \iff (3)$, which completes the proof. 
\end{proof}	
	
\begin{lemma}\label{prop_toric_Gorenstein}
	For a toric Schubert variety $X_w$, the following statements are equivalent.
	\begin{enumerate}
	\item $X_w$ is Gorenstein.
	\item $X_w$ is smooth or 
	\[
	w = s_{d+k-1} s_{d+k-2} \cdots s_{d+1} s_{d-k+1} s_{d-k+2} \cdots s_d \quad \text{ for }0 < k \leq d.
	\]
	\item The partition $\lambda_w$ is a row partition, a column partition, or a hook-shaped partition having the same arm-length and leg-length. 
	\end{enumerate}
\end{lemma}	
\begin{proof}
	Suppose that $X_w$ is a toric Schubert variety. 
	By Proposition~\ref{prop_classify_toric_Schubert}, the partition $\lambda_w$ is a hook-shaped partition. 
	There are at most two corners of $\lambda_w$. Indeed, there is one corner if $\lambda_w$ is a row partition or a column partition; two corner if $\lambda_w$ has at least one leg-length and one arm-length. 
	We depict a diagram for a row partition with one corner colored in blue below. 
	
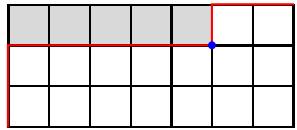
\begin{center}
	\begin{tikzpicture}[inner sep=0in,outer sep=0in]
		\node (n) {
			\begin{varwidth}{5cm}{
					\begin{ytableau}
						*(gray!30)~ & *(gray!30)~&*(gray!30)~ &*(gray!30)~ & *(gray!30)~ & ~ & ~\\
						~ & ~& ~ & ~ & ~ & ~ & ~ \\
						~ & ~& ~ & ~ & ~ & ~ & ~ \\
		\end{ytableau}}\end{varwidth}};
		\draw[thick,red] (n.south west)--
		([yshift=-0.545cm]n.north west) --
		([xshift=0.545 cm, yshift = -0.545cm]n.north west) --
		([xshift=-1.09cm, yshift=-0.545cm]n.north east) 	node[draw=none, shape = circle,fill = blue, scale= 3] {}
		--
		([xshift=-1.09cm]n.north east) --
		(n.north east);
	\end{tikzpicture}
\end{center} 
If there is one corner, that is, $\lambda_w$ is a row partition or a column partition, then it corresponds to a Gorenstein Schubert variety by Proposition~\ref{prop_smoothness}(2).

On the other hand, consider a hook-shaped diagram as depicted below. 
Moreover, we draw two antidiagonals passing through two corners in purple. 
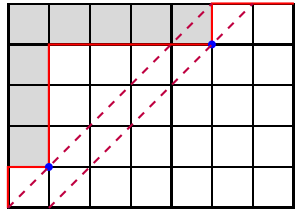
\begin{center}
	\begin{tikzpicture}[inner sep=0in,outer sep=0in]
		\node (n) {
			\begin{varwidth}{5cm}{
					\begin{ytableau}
						*(gray!30)~ & *(gray!30)~&*(gray!30)~ &*(gray!30)~ & *(gray!30)~ & ~ & ~\\
						*(gray!30)~ & ~& ~ & ~ & ~ & ~ & ~ \\
						*(gray!30)~ & ~& ~ & ~ & ~ & ~ & ~ \\
						*(gray!30)~ & ~& ~ & ~ & ~ & ~ & ~ \\
						~ & ~& ~ & ~ & ~ & ~ & ~ \\
		\end{ytableau}}\end{varwidth}};
		\draw[thick,red] (n.south west)--
		([yshift=0.545cm]n.south west) --
		([xshift=0.545cm, yshift=0.545cm]n.south west)	node[draw=none, shape = circle,fill = blue, scale= 3] {} 
		--
		([xshift=0.545 cm, yshift = -0.545cm]n.north west) --
		([xshift=-1.09cm, yshift=-0.545cm]n.north east) 	node[draw=none, shape = circle,fill = blue, scale= 3] {}
		--
		([xshift=-1.09cm]n.north east) --
		(n.north east);
		
		\draw[thick, dashed, purple] (n.south west)--([xshift=-1.09cm] n.north east);
		\draw[thick, dashed, purple] ([xshift=0.545cm] n.south west)--([xshift=-0.545cm] n.north east);
	\end{tikzpicture}
\end{center} 
In this case, we have two corners and they lie on the same antidiagonal if and only if the arm-length and the leg-length are the same. By Proposition~\ref{prop_smoothness}(2) again, this proves $(1) \iff (3)$.  

We proved in Lemma~\ref{prop_toric_smoothness} that the partition $\lambda_w$ is a row partition or a column partition if and only if $X_w$ is smooth. Moreover, when $\lambda_w$ is a hook-shaped partition having the arm-length~$k$ and the leg-length $k$, then the corresponding word $w$ is of the form 
\[
w = s_{d+k-1} s_{d+k-2} \cdots s_{d+1} s_{d-k+1} s_{d-k+2} \cdots s_d
\]
by Proposition~\ref{prop_classify_toric_Schubert}. This proves $(2) \iff (3)$, which completes the proof. 
\end{proof}

Analyzing the set of torus fixed points in a smooth toric Schubert variety, we obtain the following corollary. 
\begin{corollary}\label{cor_smooth_toric_Schubert}
A smooth toric Schubert variety $X_w$ in $\Gr(d,n)$ is isomorphic to the complex projective space $\C \P^{\ell(w)}$. 
\end{corollary}
\begin{proof}
Let $w$ be a Grassmannian permutation defining smooth toric Schubert variety. By Lemma~\ref{prop_toric_smoothness}, a reduced expression of $w$ is one of the following forms:
\begin{enumerate}
	\item $w = s_{d+x-1} s_{d+x-2} \cdots s_d$ \quad for $1 < x \leq n-d$;  or
	\item $w = s_{d-y} s_{d-y+1} \cdots s_d$ \quad for $0 \leq y \leq d-1$.
\end{enumerate}
In either case, the set of Grassmannian permutations which is less than or equal to $w$ has the cardinality $\ell(w)+1$. Accordingly, the Schubert variety $X_w$ is isomorphic to $\C \P^{\ell(w)}$ since it is the only smooth projective toric variety of dimension $\ell(w)$ having the $\ell(w)+1$ fixed points. 
\end{proof}

\begin{remark}	
	We notice that any toric Schubert variety in a Grassmannian variety can be obtained by taking the closure of a certain $T$-orbit (with respect to the left multiplication). One may consider other torus orbit closures and it is known from~\cite{NO_2019} that any smooth $T$-orbit closure in a Grassmannian variety is a product of complex projective spaces.  
\end{remark}

As is considered in Lemma~\ref{prop_toric_Gorenstein}, any Gorenstein toric Schubert variety $X_w$ is smooth or $\lambda_w = (k, 1^{k-1})$. 
To study \emph{singular} Gorenstein toric Schubert varieties, it is enough to consider the case when $k = d$ and $n=2d$, that is, $\lambda_w = (d, 1^{d-1})$ in $\Gr(d,2d)$ since the isomorphism classes of Schubert varieties can be decided by the partitions $\lambda_w$ by Proposition~\ref{prop_smoothness}(3). 
That is, it is enough to consider 
\begin{equation}\label{eq_w}
		w_d \colonequals s_{2d-1} s_{2d-2} \cdots s_{d+1} s_{1} \cdots s_d 
\end{equation}
for each positive integer $d$. 

To describe the fan of the toric Schubert variety $X_{w_d}$ in $\Gr(d,2d)$, we consider the torus fixed points in $X_{w_d}$ and their lifts in $G/B$. 

\begin{lemma}\label{lemma_Gr_permutation_v_leq_wd}
Let $w_d$ be a permutation given in~\eqref{eq_w}. Then any Grassmannian permutation $v$ satisfies $v \leq w_d$ if and only if $v=e$ or 
\[
v = s_{2d-a-1} s_{2d-a-2} \cdots s_{d+1} s_{b+1} s_{b+2} \cdots s_d
\]
for $0 \leq a \leq d-1$ and $ 0 \leq b \leq d-1$. Here, $a=d-1$ means $s_{2d-a-1} s_{2d-a-2} \cdots s_{d+1}=e$.
\end{lemma}
\begin{proof}
Since $v=e$ if and only if $\lambda_v = \emptyset$, we notice that $e < v \leq w_d$ if and only if $\emptyset <\lambda_v \leq \lambda_{{w}_d}$. In this case, since $\lambda_{{w}_d}=(d,1^{d-1})$, we have the hook-shaped partition $\lambda_v=(x,1^{y})$ such that $1 \leq x \leq d$, $0 \leq y \leq d-1$. 
By Proposition~\ref{prop_classify_toric_Schubert}, this hook-shape partition $\lambda_v=(x,1^{y})$ corresponds to
\[
v={s}_{d+x-1} s_{d+x-2} \cdots s_{d+1} {s}_{d-y} s_{d-y+1}\cdots s_d
\]
for $1 \leq x \leq d$, $0 \leq y \leq d-1$,
which completes the proof by letting $a=d-x$ and $b=d-y-1$.
\end{proof}

\begin{example}\label{example_d3_subwords}
	Let $d = 3$ and $w_d = s_5s_4s_1s_2s_3$. Varying $0 \leq a \leq 2$ and $0 \leq b \leq 3$, we have the following ten permutations which are Grassmannian permutations $v$ satisfying $v \leq w_d$.
\begin{center}
	\begin{tabular}{|c||*{3}{c|}}\hline
		\backslashbox{\makebox[1.5em]{$b$}}{$a$}
		& $0$ & $1$ & $2$ \\
		\hline\hline
		$0$ & $s_5s_4s_1s_2s_3$ & $\hati{5} s_4s_1s_2s_3$ & $\hati{5} \hati{4} s_1s_2s_3$ \\ \hline 
		$1$ & $s_5s_4\hati{1}s_2s_3$ & $\hati{5}s_4\hati{1}s_2s_3$ & $\hati{5} \hati{4}\hati{1}s_2s_3$ \\ \hline 
		$2$ & $s_5s_4\hati{1}\hati{2}s_3$ & $\hati{5}s_4\hati{1}\hati{2}s_3$ & $\hati{5} \hati{4}\hati{1}\hati{2}s_3$ \\ \hline 
		$3$ &  $e$ & $e$ & $e$ \\ \hline 
	\end{tabular}
\end{center}
In the above table, we write $\hat{s}_i$'s in gray to emphasize where we omit simple reflections.
\end{example}

Let $w = s_{i_1} \cdots s_{i_m}$ be a reduced decomposition of $w \in \mathfrak{S}_n$. For $\mathscr J = \{ 1 \leq j_1 < \dots < j_k \leq m\} \subset [m]$, we denote by $w(\mathscr J)$ the corresponding subword $s_{i_{j_1}} \cdots s_{i_{j_k}}$ of $w$. In the following lemma, we consider all possible \emph{lifts} in $\mathfrak{S}_{2d}$ of Grassmannian permutations $v$ satisfying $v \leq w_d$.
Recall that we denote $[v] = \{ u \in \mathfrak{S}_n \mid u \leq w, u W_P = v W_P\}$. 
\begin{lemma}\label{lemma_lifts_of_v}
	Let $w_d$ be a permutation given in~\eqref{eq_w}.
Let $v$ be a Grassmannian permutation satisfying $v \leq w_d$, that is, $v = e$ or 
$v = s_{2d-a-1} s_{2d-a-2} \cdots s_{d+1} s_{b+1} s_{b+2} \cdots s_d$ for $0 \leq a \leq d-1$ and $0 \leq b \leq d-1$. 
Then the following statements hold. 
\begin{enumerate}
	\item If $v \neq e$, then 
\[
[v] = \{ w_d(\mathscr{J} \cup [a+1,d-1] \cup [d+b,2d-1]) \mid \mathscr{J} \subseteq [a-1] \cup [d,d+b-2]\}.
\]
\item If $v = e$, then 
\[
[v] =  \{ w_d(\mathscr{J}) \mid \mathscr{J} \subseteq [2d-2]\}.
\]
\end{enumerate}
Here, we set $[a,a+b]=\{a,a+1,\dots,a+b\}$ for some integers $a$ and $b$. Moreover, we set $[a] = \emptyset$ if $a \leq 0$;  set $[a,a+b] = \emptyset$ if $b < 0$.
\end{lemma}
\begin{proof}
By Lemma~\ref{lemma_Gr_permutation_v_leq_wd}, we obtain the first assertion considering reduced decompositions of Grassmannian permutations $v$ satisfying $v \leq w_d$. 
Let $v = s_{2d-a-1} s_{2d-a-2} \cdots s_{d+1} s_{b+1} s_{b+2} \cdots s_d$, a Grassmannian permutation satisfying $v \leq w_d$ for $0 \leq a \leq d-1$ and $0 \leq b \leq d$. 
We first claim that 
\begin{equation}\label{eq_vWP_words_claim}
v W_P = w_d(\mathscr{J} \cup [a+1,d-1] \cup [d+b,2d-1])W_P
\end{equation}
holds for any $\mathscr{J} \subset [a-1] \cup [d,d+b-2]$. 
To describe the permutation $w_d(\mathscr{J} \cup [a+1,d-1] \cup [d+b,2d-1])$, we display elements of the set $\mathscr{J}$ in the following form.  
\[
\begin{split}
\mathscr{J} \cap [a-1] &= \{ 1 \leq i_1<i_2<\dots<i_p\leq a-1 \},  \\
\{z-d+1 \mid z \in \mathscr{J} \cap  [d,d+b-2]\} &= \{1 \leq j_1 < j_2 < \dots < j_q \leq b-1\}.
\end{split}
\]
Since $w_d( [a+1,d-1] \cup [d+b,2d-1]) = v$, the permutation $w_d(\mathscr{J} \cup [a+1,d-1] \cup [d+b,2d-1])$ can be written by 
\[
\underbrace{s_{2d-{i_1}} s_{2d-{i_2}} \dots s_{2d-{i_p}}}_{\mathscr{J} \cap [a-1]} s_{2d-a-1} s_{2d-a-2} \cdots s_{d+1} 
\underbrace{s_{j_1} s_{j_2} \dots s_{j_q}}_{\mathscr{J}\cap[d,d+b-2]} s_{b+1} s_{b+2} \cdots s_d.
\]
Here, we emphasize where we put additional simple reflections according to the choice of $\mathscr{J}$.

We are going to move the simple reflections related to $\mathscr{J}$ from left to right. For the right-most simple reflection $s_{j_q}$, since $j_q \leq b-1$ and $s_is_j = s_js_i$ for $|i-j|>1$, we have 
\[
s_{j_q} s_{b+1} s_{b+2} \cdots s_d = s_{b+1} s_{b+2} \cdots s_d s_{j_q}.
\]
Similarly, we can move the other simple reflections $s_{j_1},\dots,s_{j_{q-1}}$ to the right, so we obtain
\[
\underbrace{s_{j_1} s_{j_2} \dots s_{j_q}}_{\mathscr{J}\cap[d,d+b-2]} s_{b+1} s_{b+2} \cdots s_d
= s_{b+1} s_{b+2} \cdots s_d \underbrace{s_{j_1} s_{j_2} \dots s_{j_q}}_{\mathscr{J}\cap[d,d+b-2]}.
\]
On the other hand, we can move the simple reflection $s_{2d-i_p}$ to the the right after $s_d$ because $2d-a+1 \leq 2d-i_p \leq 2d-1$ and $s_is_j = s_js_i$ for $|i-j| >1$. Accordingly, we obtain
\[
\begin{split}
& w_d(\mathscr{J} \cup [a+1,d-1] \cup [d+b,2d-1]) \\
&\quad = 
\underbrace{s_{2d-{i_1}} s_{2d-{i_2}} \dots s_{2d-{i_p}}}_{\mathscr{J} \cap [a-1]} s_{2d-a-1} s_{2d-a-2} \cdots s_{d+1} 
\underbrace{s_{j_1} s_{j_2} \dots s_{j_q}}_{\mathscr{J}\cap[d,d+b-2]} s_{b+1} s_{b+2} \cdots s_d \\
& \quad = \underbrace{s_{2d-{i_1}} s_{2d-{i_2}} \dots s_{2d-{i_p}}}_{\mathscr{J} \cap [a-1]} s_{2d-a-1} s_{2d-a-2} \cdots s_{d+1} 
s_{b+1} s_{b+2} \cdots s_d \underbrace{s_{j_1} s_{j_2} \dots s_{j_q}}_{\mathscr{J}\cap[d,d+b-2]} \\
& \quad =  s_{2d-a-1} s_{2d-a-2} \cdots s_{d+1} 
s_{b+1} s_{b+2} \cdots s_d \underbrace{s_{2d-{i_1}} s_{2d-{i_2}} \dots s_{2d-{i_p}}}_{\mathscr{J} \cap [a-1]} \underbrace{s_{j_1} s_{j_2} \dots s_{j_q}}_{\mathscr{J}\cap[d,d+b-2]}.
\end{split}
\]
This proves~\eqref{eq_vWP_words_claim} holds.

If $v \neq e$, that is, $b < d$, then by~\eqref{eq_vWP_words_claim}, 
we get 
\begin{equation}\label{eq_v_reduced_words_inclusion}
[v] \supset \{ w_d(\mathscr{J} \cup [a+1,d-1] \cup [d+b,2d-1]) \mid \mathscr{J} \subseteq [a-1] \cup [d,d+b-2]\}
\end{equation}
for $0 \leq a \leq d-1$ and $0 \leq b \leq d-1$.

If $v = e$, that is $b =d$, then again by~\eqref{eq_vWP_words_claim}, we obtain the inclusion 
\begin{equation}
[e] \supset \{ w_d(\mathscr{J} \cup [a+1,d-1]) \mid \mathscr{J} \subseteq [a-1] \cup [d,2d-2]\}
\end{equation}
for $0 \leq a \leq d-1$. To obtain the description in the statement, we take the union of the right-hand-side varying $0\leq a \leq d-1$. Then we get 
\begin{equation}\label{eq_reduced_decompositions_of_v}
\bigcup_{a=0}^{d-1} 
\{ w_d(\mathscr{J} \cup [a+1,d-1]) \mid \mathscr{J} \subseteq [a-1] \cup [d,2d-2]\} 
= \{ w_d(\mathscr{J}) \mid \mathscr{J} \subseteq [2d-2]\}.
\end{equation}
One can check the above equality by enumerating the elements of each set since the set on the left-hand-side is contained in the set on the right-hand-side. If $a=0$, then the number of elements of the set on the left is $2^{d-1}$; otherwise, that is $2^{a-1}\cdot 2^{d-1}$. Since 
\[
2^{d-1}\left(1+\sum_{a=1}^{d-1} 2^{a-1} \right) = 2^{d-1}(1+2^{d-1}-1) = 2^{2d-2},
\]
we obtain~\eqref{eq_reduced_decompositions_of_v}. This proves 
\begin{equation}\label{eq_e_reduced_words_inclusion}
[e] \supset \{ w_d(\mathscr{J}) \mid \mathscr{J} \subseteq [2d-2]\}.
\end{equation}

To prove that the desired equality in the statement holds, that is, two sets in~\eqref{eq_v_reduced_words_inclusion} and in~\eqref{eq_e_reduced_words_inclusion} are the same, we are going to enumerate the numbers of elements of sets appearing in~\eqref{eq_v_reduced_words_inclusion} and~\eqref{eq_e_reduced_words_inclusion}. Because a reduced decomposition of $w_d$ consists of \emph{distinct} simple reflections (see~\eqref{eq_w}), any two subwords $w_d(\mathscr{I})$ and $w_d(\mathscr{J})$ of $w_d$ are different for different subsets $\mathscr{I}, \mathscr{J}\subset [2d-1]$. Accordingly, we have 
\[
\sum_{v \leq w_d} \# [v] = 
\sum_{v \leq w_d} \# \{ u \in \mathfrak{S}_n \mid u \leq w_d, u W_P = v W_P\}
= \# \{ u \in \mathfrak{S}_{n} \mid u \leq w_d\} = 2^{2d-1}
\]
since the number of subsets of $[2d-1]$ is $2^{2d-1}$. Here, we denote by $\# A$ the cardinality of a set $A$. 

To enumerate the number of elements in the set $\{ w_d(\mathscr{J} \cup [a+1,d-1] \cup [d+b,2d-1]) \mid \mathscr{J} \subseteq [a-1] \cup [d,d+b-2]\}$, it is enough to consider the number of subsets $\mathscr{J} \subset [a-1] \cup [d, d+b-2]\}$ since $w_d$ consists of distinct simple reflections. In the following table, we display the sets $[a-1] \cup [d,d+b-2]$ by varying numbers $a$ and $b$. 
\begin{center}
	\renewcommand*{\arraystretch}{1.2}
\begin{tabular}{|c||*{2}{c|}}\hline
	\backslashbox{\makebox[1.5em]{$b$}}{$a$}
	& $0$ & $1 \leq a \leq d-1$ \\
	\hline\hline
	$0$ & $\emptyset$ & $[a-1]$\\ \hline
	$1 \leq b \leq d$ & $[d,d+b-2]$ & $[a-1] \cup [d,d+b-2]$\\
	\hline 
\end{tabular}
\end{center}
Therefore, we get 
\[
\begin{split}
&\sum_{a=0}^{d-1} \sum_{b=0}^d \# (\{ \mathscr{J} \mid \mathscr{J} \subset [a-1] \cup [d,d+b-2]\}) \\
&\quad = 1 + 
\sum_{a=1}^{d-1} \# (\{ \mathscr{J} \mid \mathscr{J} \subset [a-1] \}) 
+ \sum_{b=1}^{d} \# (\{ \mathscr{J} \mid \mathscr{J} \subset [d, d+b-2] \}) \\
& \quad \quad \quad \quad + \sum_{a=1}^{d-1} \sum_{b=1}^d \# (\{ \mathscr{J} \mid \mathscr{J} \subset [a-1] \cup [d,d+b-2]\})\\
& \quad = 1 + (2^{d-1}-1) + (2^{d}-1) + (2^{d-1}-1)(2^d-1) \\
& \quad = 2^{2d-1}.
\end{split}
\]
This proves that the inclusions in~\eqref{eq_v_reduced_words_inclusion} and~\eqref{eq_e_reduced_words_inclusion} are indeed the equalities because they have the same numbers of elements, and the result follows.
\end{proof}

\begin{proposition}
	Let $d$ be a positive integer and let $w=w_d$ be the permutation in~\eqref{eq_w}.
Denote by $\mathbf{v}_1,\dots,\mathbf{v}_{2d-1}, \mathbf{w}_1,\dots,\mathbf{w}_{2d-1}$ the primitive ray generators of the fan of the toric Schubert variety~$X_{w}^B$. For $v \leq w$, the corresponding maximal cone $C_v$ in the fan of~$X_{w}$ in $\Gr(d,2d)$ is given as follows. 
\begin{enumerate}
	\item If $v = s_{2d-a-1} s_{2d-a-2} \cdots s_{d+1} s_1 \cdots s_d$ for $1 \leq a \leq d-1$, then 
	\[
	C_v = \Cone \left(\{\mathbf{v}_1\} \cup \left\{\mathbf{w}_i \mid i \in [2d-1] \setminus \{a\} \right\} \right). 
	\]
	\item If $v = s_{2d-1} s_{2d-2} \cdots s_{d+1} s_1 \cdots s_{b+1} s_{b+2} \cdots s_d$ for $1 \leq b \leq d-1$, then 
	\[
	C_v = \Cone \left( 
	\{\mathbf{v}_d\} \cup 
	\left\{ \mathbf{w}_i \mid i \in [2d-1] \setminus\{d+b-1\} \right\}
	\right).
	\]
	\item If $v = s_{2d-a-1} s_{2d-a-2} \cdots s_{d+1} s_{b+1} s_{b+2} \cdots s_d$ for $1 \leq a \leq d-1$ and $1 \leq b \leq d-1$, then 
	\[
	C_v = \Cone \left(
	\{ \mathbf{v}_1, \mathbf{v}_d \} \cup 
	\left\{
	\mathbf{w}_i \mid i \in [2d-1] \setminus \{a, d+b-1\}
	\right\}
	\right).
	\]	
	\item If $v = w$, then 
\[
 C_w = \Cone(\mathbf{w}_i \mid i \in [2d-1]).
\]
\item If $v = e$, then 
\[
 C_e = \Cone(\{\mathbf{v}_1, \mathbf{v}_d\} \cup \{\mathbf{w}_i \mid i \in [2d-2]\}).
\]
\end{enumerate}
Accordingly, for the fan $\Sigma_{w}$, the set of ray generators consists of the following $2d+1$ vectors.
\[
\{\mathbf{v}_1, \mathbf{v}_d, \mathbf{w}_1,\dots,\mathbf{w}_{2d-1}\}.
\]
\end{proposition}
\begin{proof}
By applying Proposition~\ref{prop_fan_of_toric_in_full} to $w=w_d$, we obtain the matrix representing ray generators $\mathbf{w}_1,\dots,\mathbf{w}_{2d-1}$ of the fan of the Schubert variety $X_{w}^B$ in $\flag(\C^{2d})$.
\begin{equation}
[\mathbf{w}_1 \ \cdots \ \mathbf{w}_{2d-1}] = 
	\begin{tikzpicture}[
		baseline,
		label distance=10pt 
		]

		\matrix [matrix of math nodes,left delimiter=(,right delimiter=),row sep=0.1cm,column sep=0.1cm] (m) {
			-1 & 0      & \dots & 0   & 0 & 0& \dots & 0 \\
			1 & -1 &   \dots & 0  & 0 & 0 & \dots & 0 \\
			\vdots & \ddots & \ddots & \vdots & \vdots & \vdots & \ddots & \vdots\\
			0 &     \dots & 1  & -1& 0 & 0 & \dots & 0 \\
			0 &   0    &  \dots   & 0  & -1 & 0& \dots & 0 \\
			0 &   0    &  \dots   & 0  & 1 & -1 &\dots & 0 \\
			\vdots  & \vdots & \ddots & \vdots & \vdots &\ddots & \ddots & \vdots\\
			0    & 0    &   0   & 1   & 0 & \dots & 1 & -1 \\ };
		
		\draw[dashed] ($0.5*(m-1-4.north east)+0.5*(m-1-5.north west)$) -- ($0.5*(m-8-5.south east)+0.5*(m-8-4.south west)$);
		
		\draw[dashed] ($0.5*(m-4-1.south west)+0.5*(m-5-1.north west)$) -- ($0.5*(m-4-8.south east)+0.5*(m-5-8.north east)$);
		
		\node[
		fit=(m-1-1)(m-1-4),
		inner xsep=0,
		above delimiter=\{,
		label=above:$d-1$
		] {};
		
		\node[
		fit=(m-1-5)(m-1-8),
		inner xsep=0,
		above delimiter=\{,
		label=above:$d$
		] {};
		
		\node[
		fit=(m-1-8)(m-4-8),
		inner xsep=15pt,inner ysep=0,
		right delimiter=\},
		label=right:$d-1$
		] {};
		
		\node[
		fit=(m-5-8)(m-8-8),
		inner xsep=15pt,inner ysep=0,
		right delimiter=\},
		label=right:$d$
		] {};
		
	\end{tikzpicture} 
\end{equation}
From this matrix representation, let $1 \leq i\leq 2d-1$, then we have the following relations.
\begin{eqnarray}
\label{eq_rel_ray1} 
\mathbf{v}_i+\mathbf{w}_i&=&\mathbf{v}_{i+1} \quad \text{ for } i\neq d-1, 2d-1; \\
\label{eq_rel_ray2} 
\mathbf{v}_{d-1}+\mathbf{w}_{d-1}&=&\mathbf{v}_{2d-1}; \\
\label{eq_rel_ray3}
\mathbf{v}_{2d-1}+\mathbf{w}_{2d-1}&=&\mathbf{0}. 
\end{eqnarray}

For $0 \leq a \leq d-1$ and $0 \leq b \leq d-1$, consider $v = s_{2d-a-1} s_{2d-a-2} \cdots s_{d+1} s_{b+1} s_{b+2} \cdots s_d$, which is a Grassmannian permutation less than or equal to $w$. 
Then we have 
\begin{equation}\label{eq_Cv_general_form}
	\begin{split}
	C_v &= \bigcup_{u \in [v]} C_u^{B} \quad \text{(by Proposition~\ref{prop_fan_in_Gr})} \\
	=&\bigcup_{\mathscr{J} \subseteq [a-1] \cup [d,d+b-2]} C_{w(\mathscr{J} \cup [a+1,d-1]\cup [d+b,2d-1])}^B \quad \text{(by Lemma~\ref{lemma_lifts_of_v})} \\    
	=& \bigcup_{\mathscr{J} \subseteq [a-1] \cup [d,d+b-2]}   \Cone \big( 
	\{\mathbf{v}_i \mid i \notin \mathscr{J} \cup [a+1,d-1]\cup [d+b,2d-1]\} \\
	& \qquad  \qquad \qquad\qquad \quad \qquad \cup \{\mathbf{w}_i \mid i \in \mathscr{J} \cup [a+1,d-1]\cup [d+b,2d-1] \}
	\big)  \quad \text{(by Proposition~\ref{prop_fan_of_toric_in_full})} \\
	&= \Cone \big(\{\mathbf{v}_i \mid i \in [a] \cup [d, d+b-1]  \} \\
	& \qquad \qquad \quad \cup \{\mathbf{w}_i \mid i \in [a-1] \cup [a+1,d-1] \cup [d,d+b-2] \cup [d+b, 2d-1]\} \big).
	\end{split}
\end{equation}

\textsf{Case 1.} Suppose that $v = s_{2d-a-1} s_{2d-a-2} \cdots s_{d+1} s_1 \cdots s_d$ for $1 \leq a \leq d-1$, that is, $b=0$. 
By applying $b=0$ to the equation~\eqref{eq_Cv_general_form}, we obtain
\[
C_v = \Cone \left( 
	\{ \mathbf{v}_i \mid i \in [a]  \} \cup 
	\{\mathbf{w}_i \mid i \in [a-1] \cup [a+1,2d-1]\}
	\right).
\]
Using the relation~\eqref{eq_rel_ray1}, we get 
\[
C_v = \Cone \left(\{\mathbf{v}_1\} \cup \left\{\mathbf{w}_i \mid i \in [2d-1] \setminus \{a\} \right\} \right).
\]
This proves the claim when $b=0$.

\medskip 
\textsf{Case 2.} Suppose that $v = s_{2d-1} s_{2d-2} \cdots s_{d+1} s_1 s_{b+1}s_{b+2} \cdots s_d$ for $1 \leq b \leq d-1$, that is, $a = 0$. 
By applying $a = 0$ to the equation~\eqref{eq_Cv_general_form}, we obtain
\[
C_v = \Cone \big(\{\mathbf{v}_i \mid i \in [d, d+b-1]  \} \cup \{\mathbf{w}_i \mid i \in [1,d+b-2] \cup [d+b, 2d-1]\} \big).
\]
Using the equation~\eqref{eq_rel_ray1}, we have
\[
    C_v = \Cone \left( 
	\{\mathbf{v}_{d}\} \cup 
	\left\{ \mathbf{w}_i \mid i \in [2d-1] \setminus\{d+b-1\} \right\}\right).
\]
This proves the claim when $a =0$.

\medskip 
\textsf{Case 3.} Suppose that  $v = s_{2d-a-1} s_{2d-a-2} \cdots s_{d+1} s_{b+1} s_{b+2} \cdots s_d$ for $1 \leq a \leq d-1$ and $1 \leq b \leq d-1$. By~\eqref{eq_Cv_general_form}, we get
\[
C_v = \Cone \left(\{\mathbf{v}_i \mid i \in [a] \cup [d,d+b-1] \} \cup \left\{\mathbf{w}_i \mid i \in [2d-1] \setminus \{a, d+b-1\} \right\} \right).
\] 
Using the equation~\eqref{eq_rel_ray1}, we have 
\[
    C_v = \Cone \left(
	\{ \mathbf{v}_1, \mathbf{v}_{d} \} \cup 
	\left\{
	\mathbf{w}_i \mid i \in [2d-1] \setminus \{a, d+b-1\}
	\right\}\right).
\]
This proves the claim for $1 \leq a \leq d-1$ and $1 \leq b \leq d$.

\medskip 
\textsf{Case 4.} If $v = w$, then $[w] = \{w\}$. Hence, by Propositions~\ref{prop_fan_of_toric_in_full} and~\ref{prop_fan_in_Gr}, we have
\[
 C_w = \Cone(\{\mathbf{w}_i \mid i \in [2d-1]\}).
\]

\medskip 

\textsf{Case 5.} If $v = e$, then by Propositions~\ref{prop_fan_of_toric_in_full} and~\ref{prop_fan_in_Gr} again, and Lemma~\ref{lemma_lifts_of_v}(2), we obtain
\[
C_e = \bigcup_{\mathscr{J} \subseteq [2d-2]}
\Cone \left( 
\{\mathbf{v}_i \mid i \notin \mathscr{J}\} \cup \{\mathbf{w}_i \mid i \in \mathscr{J}\} \right) = \Cone \left( \{\mathbf{v}_{i} \mid i \in [2d-2]\} \cup \left\{ \mathbf{w}_i \mid i \in [2d-2] \right\}\right).
\]
By the equations~\eqref{eq_rel_ray1}, 
we have
\begin{eqnarray*}
C_e = \Cone \left( 
	\{\mathbf{v}_{1}, \mathbf{v}_{d}\} \cup 
	\left\{ \mathbf{w}_i \mid i \in [2d-2] \right\}\right).
\end{eqnarray*}
Hence the result follows.
\end{proof}

Now we are ready to prove Theorem~\ref{thm_intro_Fano}. 
\begin{proof}[Proof of Theorem~\ref{thm_intro_Fano}]
Let $X_w$ be a toric Schubert variety in a Grassmannian $\Gr(k,n)$. If $X_w$ is smooth, then is isomorphic to the complex projective space by Corollary~\ref{cor_smooth_toric_Schubert}. Accordingly, it is Gorenstein Fano. 

Suppose that $X_w$ is not smooth but Gorenstein. Then, by Lemma~\ref{prop_toric_Gorenstein} (and also as discussed in Section~\ref{section_Gorenstein_toric_Schubert}), we may assume that $k=d$, $n = 2d$, and $w = w_d$ in~\eqref{eq_w}. 

Let $\Sigma$ be the fan of the toric Schubert variety $X_{w_d}$ in $\Gr(d,2d)$.	
Let $D = - \sum_{\rho} D_{\rho}$ be the anticanonical divisor. Since we are considering Gorenstein toric Schubert variety, we have the Cartier data $\{m_{\sigma}\}$ for the divisor $D$, that is, for each maximal cone $\sigma$, the Cartier data $\{m_{\sigma}\}$ satisfies 
\begin{equation}\label{eq_Cartier_data}
	\langle m_{\sigma}, u_{\rho} \rangle = -1 \quad \text{ for any } \rho \in \sigma(1).
\end{equation}

To prove the ampleness of $D$, it is enough to prove the following
\begin{equation}\label{eq_ampleness}
\langle m_{\sigma}, u_{\rho} \rangle > -1 \quad \text{ for all } \rho \in \Sigma(1) \setminus \sigma(1) \text{ and } \sigma \in \Sigma(n)
\end{equation}
(see~\cite[Lemma~6.1.13 and Theorem~6.1.14]{CLS11Toric}).
By~\eqref{eq_Cartier_data}, it is enough to prove that for $\rho \in \Sigma(1) \setminus \sigma(1)$ and $\sigma \in \Sigma(n)$, each $u_{\rho}$ can be expressed as the negative sum of the ray generators in $\sigma(1)$ to obtain the inequality~\eqref{eq_ampleness}. 
We prove that the inequality in~\eqref{eq_ampleness} holds by considering
each maximal cone in $\Sigma$. 

\textsf{Case 1.} If $v =s_{2d-a-1} s_{2d-a-2} \cdots s_{d+1} s_1 \cdots s_d$ for $1 \leq a \leq d-1$, then $\Sigma(1) \setminus C_v(1) = \{\mathbf{v}_d, \mathbf{w}_a\}$. 
	Then we obtain the desired expressions as follows.
	\begin{eqnarray}
		\mathbf{v}_d &=& -(\mathbf{w}_d + \mathbf{w}_{d+1} + \cdots + \mathbf{w}_{2d-1}), \label{eq_vd}\\
\mathbf{w}_a &=& -(\mathbf{v}_1 + \mathbf{w}_1 + \dots + \mathbf{w}_{a-1} + \mathbf{w}_{a+1} + \dots + \mathbf{w}_{d-1}+\mathbf{w}_{2d-1}).  \label{eq_wa}
	\end{eqnarray}

\medskip 
\textsf{Case 2.} If $v = s_{2d-1} s_{2d-2} \cdots \cdots s_{d+1} s_{b+1} s_{b+2} \cdots s_d$ for $1 \leq b \leq d-1$, then $\Sigma(1) \setminus C_v(1) = \{\mathbf{v}_1, \mathbf{w}_{d+b-1}\}$. Then we obtain
\begin{eqnarray}
	\mathbf{v}_1 &=& -(\mathbf{w}_1 +\dots + \mathbf{w}_{d-1} + \mathbf{w}_{2d-1}), \label{eq_v1}\\
	\mathbf{w}_{d+b-1} &=& -(\mathbf{v}_d+\mathbf{w}_d+\dots+\mathbf{w}_{d+b-2}+\mathbf{w}_{d+b}+\dots+\mathbf{w}_{2d-1}). \label{eq_w_db1}
\end{eqnarray}

\medskip 
\textsf{Case 3.} If $v = s_{2d-a-1} s_{2d-a-2} \cdots s_{d+1} s_{b+1} s_{b+2} \cdots s_d$ for $1 \leq a \leq d-1$ and $1 \leq b \leq d-1$, then $\Sigma(1) \setminus C_v(1) = \{\mathbf{w}_a, \mathbf{w}_{d+b-1}\}$. Then we can use the expressions in~\eqref{eq_wa} and~\eqref{eq_w_db1}. 

\medskip 
\textsf{Case 4.} If $v = w$, then $\Sigma(1) \setminus C_v(1) = \{\mathbf{v}_1, \mathbf{v}_d\}$. Then we can use the expressions in~\eqref{eq_vd} and~\eqref{eq_v1}. 

\medskip 
\textsf{Case 5.} If $v = e$, then $\Sigma(1) \setminus C_v(1) = \{\mathbf{w}_{2d-1}\}$ and we have  $\mathbf{w}_{2d-1} = -\mathbf{v}_1 - \mathbf{w}_1 - \cdots - \mathbf{w}_{d-2} - \mathbf{w}_{d-1}$, which completes the proof. 
\end{proof}

\begin{thebibliography}{10}
	
	\bibitem{BL20Singular}
	Sara Billey and V.~Lakshmibai.
	\newblock {\em Singular loci of {S}chubert varieties}, volume 182 of {\em
		Progress in Mathematics}.
	\newblock Birkh\"auser Boston, Inc., Boston, MA, 2000.
	
	\bibitem{Brion_Lecture}
	Michel Brion.
	\newblock Lectures on the geometry of flag varieties.
	\newblock In {\em Topics in cohomological studies of algebraic varieties},
	Trends Math., pages 33--85. Birkh\"{a}user, Basel, 2005.
	
	\bibitem{CLS11Toric}
	David~A. Cox, John~B. Little, and Henry~K. Schenck.
	\newblock {\em Toric varieties}, volume 124 of {\em Graduate Studies in
		Mathematics}.
	\newblock American Mathematical Society, Providence, RI, 2011.
	
	\bibitem{Fan98}
	C.~Kenneth Fan.
	\newblock Schubert varieties and short braidedness.
	\newblock {\em Transform. Groups}, 3(1):51--56, 1998.
	
	\bibitem{Fulton97Young}
	William Fulton.
	\newblock {\em Young tableaux}, volume~35 of {\em London Mathematical Society
		Student Texts}.
	\newblock Cambridge University Press, Cambridge, 1997.
	\newblock With applications to representation theory and geometry.
	
	\bibitem{GK94Bott}
	Michael Grossberg and Yael Karshon.
	\newblock Bott towers, complete integrability, and the extended character of
	representations.
	\newblock {\em Duke Math. J.}, 76(1):23--58, 1994.
	
	\bibitem{HodgesLakshmibai22}
	Reuven Hodges and Venkatramani Lakshmibai.
	\newblock A classification of spherical {S}chubert varieties in the
	{G}rassmannian.
	\newblock {\em Proc. Indian Acad. Sci. Math. Sci.}, 132(2):Paper No. 68, 27,
	2022.
	
	\bibitem{Karu13Schubert}
	Paramasamy Karuppuchamy.
	\newblock On {S}chubert varieties.
	\newblock {\em Comm. Algebra}, 41(4):1365--1368, 2013.
	
	\bibitem{LMP_Handbook}
	Eunjeong Lee, Mikiya Masuda, and Seonjeong Park.
	\newblock Torus orbit closures in the flag variety.
	\newblock arXiv:2203.16750v2.
	
	\bibitem{NO_2019}
	M.~Nodzi and K.~Ogivara.
	\newblock Smooth torus orbit closures in {G}rassmannians.
	\newblock {\em Tr. Mat. Inst. Steklova}, 305:271--282, 2019.
	\newblock English version published in Proc. Steklov Inst. Math. {\bf 305}
	(2019), no. 1, 251--261.
	
	\bibitem{TX23isomGrassmannianSchubert}
	Mihail Tarigradschi and Weihong Xu.
	\newblock The isomorphism problem for {G}rassmannian {S}chubert varieties.
	\newblock {\em J. Algebra}, 633:225--241, 2023.
	
	\bibitem{WooYong06Gorenstein}
	Alexander Woo and Alexander Yong.
	\newblock When is a {S}chubert variety {G}orenstein?
	\newblock {\em Adv. Math.}, 207(1):205--220, 2006.
	
\end{thebibliography}

\end{document}